\documentclass[11pt]{article}
\usepackage[margin=1in]{geometry}
\usepackage{amsmath,amssymb,amsthm,mathtools}
\usepackage{tikz}
\usetikzlibrary{shapes.geometric}
\usepackage{enumitem}
\usepackage{hyperref}
\usepackage[expansion=false]{microtype}
\hypersetup{
  colorlinks=true,
  linkcolor=blue,
  citecolor=blue,
  urlcolor=blue,
  pdftitle={A Polynomial Recovery Criterion for Forced Commutativity in the Matrix Square-Root Fiber of a Square-Free Cubic Polynomial},
  pdfauthor={Oleg Kiriukhin}
}

\theoremstyle{plain}
\newtheorem{theorem}{Theorem}[section]
\newtheorem{lemma}[theorem]{Lemma}
\newtheorem{proposition}[theorem]{Proposition}

\theoremstyle{definition}
\newtheorem{definition}[theorem]{Definition}
\newtheorem{example}[theorem]{Example}

\theoremstyle{remark}
\newtheorem{remark}[theorem]{Remark}

\DeclareMathOperator{\Jord}{Jord}
\DeclareMathOperator{\im}{im}
\DeclareMathOperator{\Spec}{Spec}

\title{A Polynomial Recovery Criterion for Forced Commutativity in the Matrix Square-Root Fiber of a Square-Free Cubic Polynomial}
\author{%
Oleg Kiriukhin\\
City University of Hong Kong\\
\texttt{okiriukh@cityu.edu.hk}%
}
\date{\today}

\begin{document}
\maketitle

\begin{abstract}
Let $k$ be an algebraically closed field of characteristic different from $2$ and $3$, and let
$f(t)=t^3+at+b\in k[t]$ be square-free.  For $X\in M_m(k)$ set $A=f(X)$.  This paper studies the
full matrix equation $Y^2=A$, without imposing $[X,Y]=0$.  Existence is the classical square-root
problem for the fixed matrix $A$.  The only obstruction is the nilpotent square-root pairing
criterion on the zero-primary component of $A$.  The main result is a fiberwise commutativity
classification: if the fiber $\{Y\in M_m(k):Y^2=A\}$ is nonempty, then every such $Y$ commutes with
$X$ if and only if $X\in k[A]$.  The forward direction restricts the classical centralizer
criterion of R.\,C. Thompson \cite{Thompson1969} to the square-root fiber.  The new content is the
converse: when $X\notin k[A]$, an explicit square root $Y$ of $A$ with $[X,Y]\ne 0$ is exhibited.
The witnesses arise from nonzero spectral collision under $f$, critical nilpotent collapse, or
a multicolor zero-primary fiber.
\end{abstract}



\section{Introduction}\label{sec:introduction}
Let $k$ be an algebraically closed field with $\operatorname{char} k\ne 2,3$, let $f(t)=t^3+at+b\in
k[t]$ be square-free, and fix $X\in M_m(k)$ with $A=f(X)$.  This paper studies when $X$ is
\emph{polynomially recoverable} from its $f$-image, that is, when there exists $g\in k[t]$ with
$X=g(A)$, equivalently $X\in k[A]$.  The main theorem characterizes polynomial recovery in
fiberwise terms: assuming the matrix square-root fiber $\{Y:Y^2=A\}$ is nonempty, polynomial
recovery is equivalent to the condition that every $Y$ in this fiber commutes with $X$.  The
background question is thus a polynomial-recovery question for the map $X\mapsto f(X)$ on
$M_m(k)$, expressed combinatorially through the geometry of the square-root fiber.

The scalar equation $y^2=f(x)$ is a short-Weierstrass equation for a smooth elliptic curve after
adjoining the point at infinity, and provides the formal motivation for studying the matrix
equation $Y^2=f(X)$, but no commutativity is imposed and the elliptic geometry plays no role in
the statement or proof. Equivalently, a pair $(X,Y)\in M_m(k)^2$ with $Y^2=f(X)$ is a finite-dimensional representation of the noncommutative $k$-algebra $\nu_f=k\langle x,y\rangle/(y^2-f(x))$, and the main theorem classifies the forced-commutativity locus inside the representation variety $\operatorname{Rep}_m(\nu_f)$.

The classification rests on the following structural observation.  If $A=f(X)$ and $Y^2=A$, then $Y$ commutes with
$A$, hence lies in $Z(A)$, while the desired conclusion is membership in $Z(X)$.  The main theorem
identifies the exact point at which these two requirements coincide on the square-root fiber: after
existence is known, all square roots commute with $X$ precisely when $X\in k[A]$.

The proof separates two issues.  Existence of a square root is the classical problem for the fixed
matrix $A$, with the only obstruction occurring on the zero-primary component.  Forced commutativity
is a different, relative question: it depends on whether the map $X\mapsto f(X)$ has lost spectral
or nilpotent information needed to recover $X$ from $f(X)$.

\noindent\textbf{Main theorem.}  Let $X\in M_m(k)$ and set $A=f(X)$.  Assume that the fiber $\{Y\in
M_m(k):Y^2=A\}$ is nonempty.  Then every $Y$ satisfying $Y^2=A$ commutes with $X$ if and only if
$X\in k[A]$.\footnote{Throughout the paper, ``square root'' means a matrix $Y$ with $Y^2=A$.  This
is the matrix-equation sense studied for example in
\cite{HighamFunctions,JohnsonOkuboReams,BjorkHammarling}, and is distinct from a matrix
factorization $\Phi\Psi=AI$ in the sense of \cite{EisenbudMF}.}  Equivalently, if $X\notin k[A]$,
then the fiber contains a square root $Y$ with $[X,Y]\ne 0$.

\begin{definition}[Polynomial recovery]\label{def:polynomial-recovery}
$X$ is \emph{polynomially recoverable} from $A=f(X)$, equivalently \emph{polynomial recovery holds}
for the pair $(X,f)$, if $X\in k[A]$.  Equivalently, the
evaluation map $k[t]\to M_m(k)$ sending $t\mapsto A$ has image $k[A]$ containing $X$, that is,
there exists $g\in k[t]$ with $X=g(A)$.\footnote{The condition $X\in k[A]=k[f(X)]$ is the
centralizer-equality condition $Z(A)=Z(X)$ in the classical sense of R.\,C. Thompson
\cite{Thompson1969}, restricted to the special case $A=f(X)$.  The terminology ``polynomial
recovery'' is introduced here to emphasize the explicit inversion $X=g(A)$ and to distinguish
the relative question of recovery under the map $X\mapsto f(X)$ from the absolute square-root
existence question for the fixed matrix $A$.}
\end{definition}

\noindent\textbf{Relation to the classical centralizer criterion.}  The forward direction of the
main theorem follows from the classical result of R.\,C. Thompson \cite{Thompson1969}, which
establishes that the centralizer $Z(A)$ is contained in $Z(X)$ if and only if $X\in k[A]$.
Applied to a square root $Y$ of $A=f(X)$, Lemma~\ref{lem:centralizer-containment} gives
$Y\in Z(A)$, and Thompson's theorem then forces $Y\in Z(X)$ when $X\in k[A]$.  The new content of
this paper is the converse on the square-root fiber: when $X\notin k[A]$, an explicit square root
$Y$ of $A$ with $[X,Y]\ne 0$ is exhibited.  Thompson's theorem produces some element of
$Z(A)\setminus Z(X)$, but in general this element is not a square root of $A$.  The constructive
refinement carried out below is the contribution.

\noindent\textbf{Proof strategy.}  Polynomial recovery can fail only through spectral collision
under $f$ or through a nontrivial nilpotent block at a critical point of $f$.  A zero-valued
collision is treated by retaining the root-color data in the singular fiber.  A nonzero collision
or a critical collapse produces an idempotent in $Z(A)$ which does not commute with $X$, and an
idempotent twist then gives the required noncommuting square root.  The final classification is
obtained by combining these constructions with the square-root existence criterion.

\noindent\textbf{Organization.}  Section~\ref{sec:setup} fixes notation and records the basic
centralizer containment.  Section~\ref{sec:scope-main-results} states the classification in its
fiberwise form.  Sections~\ref{sec:polynomial-square-roots-invertible}--\ref{sec:partition-transforms}
collect the square-root, idempotent, and Jordan-partition tools.
Section~\ref{sec:colored-singular-fibers} constructs
the noncommuting witnesses.  Section~\ref{sec:global-classification-after-existence} proves the
global theorem after existence.  When $f$ has no critical points or collision eigenvalues among
the eigenvalues of $X$, polynomial recovery holds and the classification is immediate.
The regular case is therefore subsumed by the main theorem.  The final sections give further directions, situate
the result in the literature, and record illustrative examples.

\section{Setup}\label{sec:setup}
Throughout the paper, $k$ is an algebraically closed field with $\operatorname{char} k\ne 2,3$, and
\[
        f(t)=t^3+at+b\in k[t]
\]
is square-free, equivalently $4a^3+27b^2\ne 0$.  For a matrix $T$, write $Z(T)$ for its centralizer
in the full matrix algebra and $k[T]$ for the subalgebra generated by $T$.  All commutators are
written as $[S,T]=ST-TS$.  For an eigenvalue $\lambda\in\Spec(X)$, write $V_\lambda$ for the
generalized $\lambda$-eigenspace of $X$ and $N_\lambda=(X-\lambda I)|_{V_\lambda}$ for the nilpotent
part of $X$ on $V_\lambda$, so that $X|_{V_\lambda}=\lambda I+N_\lambda$.  The full primary
decomposition is $V=\bigoplus_{\lambda\in\Spec(X)}V_\lambda$.

The word elliptic refers only to the underlying scalar short-Weierstrass equation $y^2=f(x)$.  After
adjoining the point at infinity this scalar equation defines a smooth elliptic curve, whereas this
paper studies the associated matrix equation without imposing commutativity of the two
matrix coordinates.

For $m\ge 1$ write $\mathcal M_m(f)=\{(X,Y)\in M_m(k)^2:Y^2=f(X)\}$ for the matrix solution
locus, $\mathcal C_m(f)=\{(X,Y)\in\mathcal M_m(f):[X,Y]=0\}$ for its commuting sublocus, and
$\mathcal{NC}_m(f)=\mathcal M_m(f)\setminus\mathcal C_m(f)$ for the noncommuting locus.  In the
representation-variety language, $\mathcal M_m(f)$ is the variety of $m$-dimensional
representations of the noncommutative $k$-algebra $\nu_f=k\langle x,y\rangle/(y^2-f(x))$, and
$\mathcal C_m(f)$ is its abelianization, the variety of $m$-dimensional representations of the
commutative coordinate ring $k[x,y]/(y^2-f(x))$ of the affine elliptic curve.

\begin{lemma}[Centralizer containment]\label{lem:centralizer-containment}
If $(X,Y)\in\mathcal M_m(f)$ and $A=f(X)$, then $Y\in Z(A)$.
\end{lemma}

\begin{proof}
Since $A=Y^2$,
\[
 YA=YY^2=Y^3=Y^2Y=AY.
\]
Thus $Y$ commutes with $A$.
\end{proof}

\begin{remark}[Link to the centralizer-equivalence literature]\label{rem:centralizer-equivalence-link}
Let $X\in M_m(k)$ and $A=f(X)$.  Thompson's theorem \cite{Thompson1969} states that over an
algebraically closed field, $Z(A)\subseteq Z(X)$ if and only if $X\in k[A]$, and subsequent work
has refined this polynomial-equivalence relation in various directions
\cite{XiZhangCentralizer,ZhangZhuPolynomialEquivalence}.  Combined with
Lemma~\ref{lem:centralizer-containment}, Thompson's theorem gives the forward direction of the
main theorem directly: if $X\in k[A]$ then every $Y$ with $Y^2=A$ lies in $Z(A)\subseteq Z(X)$ and
hence commutes with $X$.  The contribution of this paper is the converse over the
square-root fiber: when $X\notin k[A]$ and the fiber $\{Y:Y^2=A\}$ is nonempty, an explicit
square root $Y$ with $[X,Y]\ne 0$ is constructed.
\end{remark}

\medskip
\noindent\textbf{Standing hypotheses and notation.}  Throughout the remainder of the paper, $X\in
M_m(k)$ is fixed and $A=f(X)$ is its $f$-image.  The standing assumptions $\operatorname{char}
k\ne 2,3$, $k=\bar k$, and $4a^3+27b^2\ne 0$ (equivalently, $f$ square-free) are in force without
further mention.  The fiberwise hypothesis ``the fiber $\{Y:Y^2=A\}$ is nonempty'' is imposed only
where stated.  The noncommuting constructions in
Section~\ref{sec:colored-singular-fibers} either
operate on an invertible $A$-component, where existence is automatic by
Lemma~\ref{lem:poly-square-root}, or take the global fiber to be nonempty and combine with the
localized idempotent twist (Lemma~\ref{lem:localized-idempotent-twist}).

\begin{remark}[Why $\operatorname{char} k\ne 2,3$]\label{rem:characteristic-restrictions}
The two exclusions enter the paper for independent reasons.  The hypothesis $\operatorname{char}
k\ne 2$ is used twice: the binomial series $(1+z)^{1/2}$ underlying the polynomial square root in
Lemma~\ref{lem:poly-square-root} requires the scalar $1/2\in k$, and the reflection $R=2P-I$ that
drives the idempotent twist of Lemma~\ref{lem:idempotent-twist} satisfies $R=-I$ in characteristic
$2$, so $Y=RG=-G$ commutes with $X$ whenever $G$ does and the twist collapses.  The hypothesis
$\operatorname{char} k\ne 3$ is used once, in Lemma~\ref{lem:critical-collapse-reduction}: the
critical-collapse exponent $d\in\{2,3\}$ from
Lemma~\ref{lem:polynomial-recovery-dichotomy} requires the cube-root series $(1+z)^{1/3}$ in the
case $d=3$, which occurs only for $f(t)=t^3+b$ with critical point $\lambda=0$.  In characteristic
$3$ the cubic itself degenerates further: $f''(\lambda)=6\lambda\equiv 0$, so every critical point
is automatically of higher order and the dichotomy structure changes.  Both restrictions could in
principle be relaxed by replacing the binomial series with Hensel or Newton iteration
(cf.\ \cite[Ch.\,5--6]{HighamFunctions}) and the
reflection twist with an alternative construction.  This is not pursued here.
\end{remark}

\section{Scope and main results}\label{sec:scope-main-results}

The classification has two layers.  First, for a fixed $X$, existence of a solution to $Y^2=f(X)$ is
the classical square-root problem for the matrix $f(X)$.  The only obstruction is the nilpotent
square-root pairing criterion on the zero-primary component.  Second, once the fiber is nonempty, forced
commutativity is governed by polynomial recovery of $X$ from $f(X)$.

\begin{theorem}[Fiberwise forced-commutativity classification]\label{thm:fiberwise-forced-commutativity-intro}
Let $X\in M_m(k)$ and set $A=f(X)$.  Assume that the fiber $\{Y\in M_m(k):Y^2=A\}$ is nonempty.
Then the following are equivalent:
\begin{enumerate}[label=\textup{(\roman*)}]
\item every $Y$ satisfying $Y^2=A$ commutes with $X$,
\item $X\in k[A]$.
\end{enumerate}
Equivalently, whenever $X\notin k[A]$, the fiber contains a square root $Y$ with $[X,Y]\ne 0$.
\end{theorem}

This theorem is proved in Section~\ref{sec:global-classification-after-existence}, after the
existence criterion and the noncommuting constructions have been established.

\section{Polynomial square roots of invertible matrices}\label{sec:polynomial-square-roots-invertible}

The following lemma is a standard consequence of the polynomial functional calculus for matrices,
for which see \cite{HighamFunctions,JohnsonOkuboReams,BjorkHammarling,HNonnegativeSquareRoots}.

\begin{lemma}[Polynomial square root of an invertible matrix]\label{lem:poly-square-root}
Let $A\in M_m(k)$ be invertible.  Then there exists a polynomial $G\in k[A]$ such that
\[
 G^2=A.
\]
Moreover, $G$ is invertible.
\end{lemma}

\begin{proof}
Let the distinct eigenvalues of $A$ be $c_1,\ldots,c_s$, all nonzero, and on the generalized
eigenspace for $c_i$ write $A=c_iI+Q_i$ with $Q_i$ nilpotent.  Choose $\gamma_i\in k$ with
$\gamma_i^2=c_i$.  The truncated binomial expansion of $\gamma_i(1+Q_i/c_i)^{1/2}$ gives a
polynomial $h_i\in k[t]$ with $h_i(A)^2=A$ on the $c_i$-primary component.  Patching by the
Chinese remainder theorem produces a single polynomial $h\in k[t]$ with $G=h(A)$ and $G^2=A$.
Since $G^2$ is invertible, so is $G$.
\end{proof}

\section{The idempotent twist}\label{sec:idempotent-twist}

The construction below converts an idempotent in $Z(A)\setminus Z(X)$ into a square root of
$A=f(X)$ that does not commute with $X$, providing the building block used throughout the proof.

\begin{lemma}[Idempotent twist]\label{lem:idempotent-twist}
Assume $A=f(X)$ is invertible.  Suppose there exists an idempotent $P\in Z(A)$ such that $P\notin
Z(X)$.  Then there exists $Y\in M_m(k)$ such that
\[
 Y^2=f(X),\qquad [X,Y]\ne 0.
\]
\end{lemma}

\begin{proof}
By Lemma~\ref{lem:poly-square-root}, choose a polynomial square root $G\in k[A]$ with $G^2=A$.  Set
\[
 R=2P-I.
\]
Then $R^2=I$, and $R\in Z(A)$ because $P\in Z(A)$.  Hence $R$ commutes with $G$.  Define
\[
 Y=RG.
\]
Then
\[
 Y^2=RGRG=R^2G^2=A=f(X).
\]
Since $G\in k[A]=k[f(X)]\subseteq k[X]$ and $A$ is invertible, $G$ is invertible and commutes with
$X$.  Therefore $[X,G]=0$, and the Leibniz rule for the commutator yields
\[
 [X,Y]=[X,RG]=[X,R]G+R[X,G]=[X,R]G.
\]
Because $P\notin Z(X)$ and $R=2P-I$, it follows that $R\notin Z(X)$, so $[X,R]\ne 0$.  Since $G$ is invertible, $[X,Y]\ne 0$.
\end{proof}

\begin{lemma}[Localized idempotent twist]\label{lem:localized-idempotent-twist}
Let $A=f(X)$, and let $W$ be a direct sum of $A$-primary components on which $A_W=A|_W$ is
invertible.  Assume that there exists an idempotent $P\in Z(A_W)$ such that $P\notin Z(X_W)$.  If
the global equation $Y^2=A$ has at least one solution, then there exists a global solution
$\widetilde Y^2=A$ such that $[X,\widetilde Y]\ne 0$.
\end{lemma}

\begin{proof}
Both $W$ and its $A$-primary complement $W'$ are $X$-stable because $A=f(X)$ preserves the primary
decomposition of $X$.  By Lemma~\ref{lem:poly-square-root} on $W$ there is a polynomial square
root $G\in k[A_W]\subseteq k[X_W]$ with $G^2=A_W$.  Setting $R=2P-I_W$ and $\widetilde Y_W=RG$ gives
$\widetilde Y_W^2=A_W$.  Since $[X_W,G]=0$, the Leibniz rule yields $[X_W,\widetilde Y_W]=[X_W,R]G$,
which is nonzero because $G$ is invertible and $R$ does not commute with $X_W$.  Choose any global
square root $Y_0^2=A$ (existence by hypothesis).  Then $Y_0$ commutes with $A$ and so preserves
$W'$, so $\widetilde Y=\widetilde Y_W\oplus Y_0|_{W'}$ satisfies $\widetilde Y^2=A$ and
$[X,\widetilde Y]\ne 0$.
\end{proof}

\section{Auxiliary Jordan lemmas}\label{sec:auxiliary-jordan-lemmas}

This section records the elementary Jordan-form facts used later.  Standard references for Jordan
form and matrix analysis include \cite{HornJohnsonMatrixAnalysis,Gantmacher}.  They are included
here to make the paper logically self-contained up to the standard existence criterion for square
roots of nilpotent matrices.

It is a standard fact (\cite[Ch.\,VI]{Gantmacher}, \cite[\S3.2]{HornJohnsonMatrixAnalysis}) that
if $N\in M_m(k)$ is nilpotent and $U(N)\in k[N]$ is invertible, then $N U(N)$ has the same Jordan
partition as $N$.  Indeed $U(N)^j$ commutes with $N^j$ and is invertible, so
$\ker(NU(N))^j=\ker N^j$ for every $j\ge 1$, and the Jordan partition of a nilpotent operator is
determined by the kernel-dimension sequence.

\begin{lemma}[Power of a nilpotent Jordan block]\label{lem:power-nilpotent-jordan-block}
Let $J_n=J_n(0)$ and let $r\ge 1$.  Write
\[
 n=qr+s,
 \qquad 0\le s<r.
\]
Then $J_n^r$ has Jordan block sizes
\[
 (q+1)^s q^{r-s},
\]
where the notation means that $q+1$ occurs $s$ times and $q$ occurs $r-s$ times, after deleting zero parts.
\end{lemma}

\begin{proof}
Let $e_1,\ldots,e_n$ be a Jordan chain for $J_n$, so that $J_ne_i=e_{i+1}$ for $i<n$ and $J_ne_n=0$.
The operator $J_n^r$ maps
\[
 e_i\mapsto e_{i+r}.
\]
Thus the chain splits into residue-class chains
\[
 e_i,e_{i+r},e_{i+2r},\ldots
\]
for $1\le i\le\min(r,n)$.  If $n=qr+s$, then exactly $s$ of these chains have length $q+1$ and the
remaining $r-s$ chains have length $q$, with zero-length chains omitted.  These chain lengths are
precisely the Jordan block sizes of $J_n^r$.
\end{proof}

\begin{lemma}[Power-collapse idempotent]\label{lem:power-collapse-idempotent}
Let $M$ be nilpotent, let $r\ge 2$, and assume that $M$ has at least one Jordan block of size at
least $2$.  Then there exists an idempotent $P$ such that
\[
 PM^r=M^rP
 \qquad\text{but}\qquad
 PM\ne MP.
\]
\end{lemma}

\begin{proof}
Choose a Jordan chain $e_0,e_1,\ldots,e_{n-1}$ for $M$ with $n\ge 2$ and $Me_i=e_{i+1}$.  Decompose
this chain into residue classes modulo $r$:
\[
 E_j=\langle e_i:i\equiv j\pmod r\rangle.
\]
Each $E_j$ is invariant under $M^r$, because $M^r$ preserves residues modulo $r$.  Let $P$ be the
projection onto one nonzero residue summand, for instance $E_0$, along the direct sum of the
remaining residue summands and all other Jordan chains.  Then $P^2=P$ and $PM^r=M^rP$.  But $M$
sends $e_0$ to $e_1$, which lies in a different residue class because $r\ge 2$ and $n\ge 2$.  Thus
$PMe_0=0$, while $MPe_0=e_1$, so $PM\ne MP$.
\end{proof}

\begin{lemma}[Critical collapse reduction]\label{lem:critical-collapse-reduction}
Let $N$ be nilpotent and let
\[
 H(N)=N^rU(N),
\]
where $r\ge 2$ and $U(N)\in k[N]$ is invertible.  Assume $\operatorname{char} k$ does not divide $r$.
Then there exists a nilpotent matrix
\[
 M=N V(N),
\]
with $V(N)\in k[N]$ invertible, such that
\[
 H(N)=M^r
 \qquad\text{and}\qquad
 k[M]=k[N].
\]
Consequently, if $N$ has a Jordan block of size at least $2$, then there is an idempotent $P$
commuting with $H(N)$ but not with $N$.
\end{lemma}

\begin{proof}
Since $U(N)$ is invertible, $U(0)\ne 0$. Choose $\alpha\in k$ with $\alpha^r=U(0)$, and write
\[
 U(N)=\alpha^r(I+W(N)),
\]
where $W(N)$ is nilpotent. Since $\operatorname{char} k$ does not divide $r$, the binomial series
$(1+z)^{1/r}$ is defined and truncates on nilpotent inputs. Hence
\[
 V(N)=\alpha(I+W(N))^{1/r}\in k[N]
\]
is invertible and satisfies $V(N)^r=U(N)$. Set $M=NV(N)$.  Then
\[
 M^r=N^rV(N)^r=N^rU(N)=H(N).
\]
Since $V(0)\ne 0$, the polynomial $tV(t)$ is invertible under composition in the truncated algebra
$k[t]/(t^{\nu})$ where $\nu$ is the nilpotency index of $N$, so $k[M]=k[N]$.  Lemma~\ref{lem:power-collapse-idempotent}
then yields an idempotent $P$ commuting with $M^r=H(N)$ but not with $M$, hence not with $N$.
\end{proof}

\begin{lemma}[Two-chain nilpotent square-root construction]\label{lem:two-chain-square-root-construction}
Let $T$ be nilpotent with two Jordan chains
\[
 u_0,u_1,\ldots,u_{p-1},
 \qquad
 v_0,v_1,\ldots,v_{q-1},
\]
where $Tu_i=u_{i+1}$ and $Tv_i=v_{i+1}$, with terminal vectors sent to $0$.  If $p=q$ or $p=q+1$,
then $T$ has a square root on the direct sum of these two chains.
\end{lemma}

\begin{proof}
Assume first $p=q$.  Define $S$ by
\[
 S u_i=v_i\quad(0\le i\le p-1),
 \qquad
 S v_i=u_{i+1}\quad(0\le i\le p-2),
 \qquad
 S v_{p-1}=0.
\]
Then $S^2u_i=u_{i+1}=Tu_i$ for $i<p-1$, and $S^2u_{p-1}=0=Tu_{p-1}$.  Also $S^2v_i=v_{i+1}=Tv_i$ for
$i<p-1$, and $S^2v_{p-1}=0=Tv_{p-1}$.

If $p=q+1$, define
\[
 S u_i=v_i\quad(0\le i\le q-1),
 \qquad
 S u_q=0,
 \qquad
 S v_i=u_{i+1}\quad(0\le i\le q-1).
\]
The same direct verification gives $S^2=T$ on both chains.
\end{proof}

\begin{proposition}[Nilpotent square-root pairing criterion, classical]\label{prop:nilpotent-square-root-criterion}
Let $T$ be nilpotent with Jordan block sizes
\[
 \theta_1\ge\theta_2\ge\cdots\ge\theta_s.
\]
Then $T$ has a square root if and only if the multiset of Jordan block sizes can be decomposed into
pairs $(p,q)$ with $|p-q|\le 1$ together with possibly any number of unpaired blocks of size exactly
$1$ (no unpaired block of size $\ge 2$ is permitted).  This is the standard nilpotent square-root
pairing criterion, going back to Cross and Lancaster \cite{CrossLancaster1974}. See also
\cite{HighamFunctions,HornJohnsonMatrixAnalysis,Gantmacher,HNonnegativeSquareRoots}.  A proof is
included for self-containment.
\end{proposition}

\begin{proof}
For sufficiency, pair the Jordan chains according to the stated rule, reorder each pair so that its
sizes are $(p,q)$ with $p=q$ or $p=q+1$, and apply the two-chain construction to each pair.  On an
unpaired size-one block, define the square root to be zero.  Taking the direct sum of these square
roots gives an operator $S$ with $S^2=T$.

For necessity, suppose $S^2=T$.  Decompose $S$ into nilpotent Jordan chains.  The square of a single
Jordan block $J_n(0)$ has block sizes obtained from Lemma~\ref{lem:power-nilpotent-jordan-block}
with $r=2$, namely either $(q,q)$ when $n=2q$ or $(q+1,q)$ when $n=2q+1$.  Thus every nontrivial
Jordan block of $T=S^2$ arises from pairing two blocks whose sizes differ by at most $1$, while a
one-dimensional zero block may arise as the square of a one-dimensional zero block.  Summing over
the Jordan blocks of $S$ gives the stated pairing condition.
\end{proof}

\section{Partition transforms}\label{sec:partition-transforms}

This section records how Jordan partitions transform under $f(X)$, in preparation for the
colored-fiber analysis of Section~\ref{sec:colored-singular-fibers}.

For a nilpotent matrix $N$, let $\Jord(N)$ denote the partition of its Jordan block sizes.  If
$\eta$ is a partition, define $\eta^{[r]}$ by replacing each part $n$ by the Jordan block sizes of
$J_n(0)^r$.  Explicitly, if
\[
 n=qr+s,\qquad 0\le s<r,
\]
then $n$ is replaced by
\[
 (q+1)^s q^{r-s},
\]
after deleting zero parts.

For each $\lambda\in\Spec(X)$, set $\eta_\lambda=\Jord(N_\lambda)$ and define $\pi_\lambda$ by
\[
\pi_\lambda=
\begin{cases}
\eta_\lambda, & f'(\lambda)\ne 0,\\
\eta_\lambda^{[2]}, & f'(\lambda)=0,\ \lambda\ne 0,\\
\eta_\lambda^{[3]}, & f(t)=t^3+b,\ \lambda=0.
\end{cases}
\]

\begin{proposition}[Local Jordan transform]\label{prop:local-jordan-transform}
The partition $\pi_\lambda$ is the Jordan partition of $f(X)|_{V_\lambda}$ at eigenvalue $f(\lambda)$.
\end{proposition}

\begin{proof}
Taylor expansion gives
\[
 f(\lambda I+N_\lambda)=f(\lambda)I+f'(\lambda)N_\lambda+3\lambda N_\lambda^2+N_\lambda^3.
\]
If $f'(\lambda)\ne 0$, the nilpotent part is $N_\lambda U(N_\lambda)$ with $U(N_\lambda)$
invertible, so the Jordan type is unchanged.  If $f'(\lambda)=0$ and $\lambda\ne 0$, the nilpotent
part is $N_\lambda^2(3\lambda I+N_\lambda)$, and $3\lambda I+N_\lambda$ is invertible, so the Jordan
type is that of $N_\lambda^2$.  If $\lambda=0$ and $f'(0)=0$, then $a=0$ and the nilpotent part is
$N_\lambda^3$.  The stated partition transform follows from
Lemma~\ref{lem:power-nilpotent-jordan-block}.
\end{proof}

\section{Colored singular fibers}\label{sec:colored-singular-fibers}

This section records the internal structure of the singular fiber $f(X)=0$.  Its purpose is to make
explicit the information that is lost when the three root-primary sectors of $X$ collapse to the
same zero-primary sector of $f(X)$.  Sections~\ref{subsec:colored-formal-data}--\ref{subsec:colored-existence-and-noncomm}
are the formal core used in the proof of the main theorem.  Section~\ref{subsec:colored-graph}
defines the color graph and gives an illustrative example.

Let $r_1,r_2,r_3$ be the distinct roots of $f$.  Since $f$ is square-free, $f'(r_i)\ne 0$ for all
$i$, so the zero-primary partition of $f(X)$ is
\[
 \theta=\eta_{r_1}\sqcup\eta_{r_2}\sqcup\eta_{r_3}.
\]

\begin{definition}[Square-root pairable partition]\label{def:square-root-pairable-partition}
A partition $\theta$ is square-root pairable if its multiset of parts can be decomposed into pairs
$(p,q)$ with $|p-q|\le 1$, together with possible singleton parts, and every singleton part is equal
to $1$.
\end{definition}

\subsection{Colored data and matchings}\label{subsec:colored-formal-data}

Let $r_1,r_2,r_3$ be the distinct roots of $f$.  On the zero-primary component of $f(X)$, write
\[
V_0=V_{r_1}\oplus V_{r_2}\oplus V_{r_3},
\]
where $V_{r_i}$ is the generalized $r_i$-eigenspace of $X$, possibly zero.  A Jordan block of the
nilpotent operator $f(X)|_{V_0}$ therefore has two labels: its size $p$ and its root-color $i$.

\begin{definition}[Colored zero partition]\label{def:colored-zero-partition}
The colored zero partition of $X$ is the multiset
\[
\mathcal P_0(X)=\{(p,i):\text{there is a size }p\text{ Jordan block of }f(X)|_{V_{r_i}}\}.
\]
Its uncolored shadow is the ordinary partition obtained by forgetting $i$.
\end{definition}

\begin{definition}[Admissible colored matching]\label{def:admissible-colored-matching}
An admissible matching of $\mathcal P_0(X)$ is a decomposition of its elements into pairs
\[
(p,i)\sim(q,j),\qquad |p-q|\le 1,
\]
together with possible singleton blocks of size $1$.  A matched edge is internal if $i=j$ and cross-colored if $i\ne j$.
\end{definition}

\subsection{Existence and color-breaking direction}\label{subsec:colored-existence-and-noncomm}

\begin{proposition}[Existence ignores color]\label{prop:existence-ignores-color}
The equation $Y^2=f(X)$ has a solution on the zero-primary component if and only if the uncolored
shadow of $\mathcal P_0(X)$ admits an admissible matching.
\end{proposition}

\begin{proof}
The claim is precisely Proposition~\ref{prop:nilpotent-square-root-criterion} applied to the nilpotent
operator $f(X)|_{V_0}$.  That criterion depends only on Jordan block sizes of $f(X)|_{V_0}$, not on
which root-primary sector of $X$ produced each block.
\end{proof}

\begin{proposition}[Cross-colored matching forces noncommutativity]\label{prop:cross-colored-matching}
If an admissible matching contains a cross-colored edge, then there exists a square root $Y$ of
$f(X)$ such that $[X,Y]\ne 0$.
\end{proposition}

\begin{proof}
Let $(p,i)\sim(q,j)$ be a cross-colored matched pair with $i\ne j$.  The construction of
Lemma~\ref{lem:two-chain-square-root-construction} defines a square root on the direct sum of the
corresponding Jordan chains for $f(X)$.  This square root maps a vector from $V_{r_i}$ to
$V_{r_j}$.  Since any operator commuting with $X$ preserves the generalized eigenspaces of $X$, and
$V_{r_i}\ne V_{r_j}$, the commutator with $X$ is nonzero.  Extending the construction over the
remaining matched blocks gives the desired global square root.
\end{proof}

\begin{theorem}[Color-breaking direction]\label{thm:color-breaking-direction}
Assume that at least two colors occur in $\mathcal P_0(X)$ and that $Y^2=f(X)$ has at least one
solution.  Then the square-root fiber contains a solution $Y$ with $[X,Y]\ne 0$.
\end{theorem}

\begin{proof}
Let $Y$ be a global solution of $Y^2=f(X)$, whose existence is part of the hypothesis.  Since $Y$
commutes with $f(X)$ and hence preserves the zero-primary component $V_0$, the restriction
$S_0:=Y|_{V_0}$ is a square root of $f(X)|_{V_0}$.  If $S_0$
sends some vector from one root-color space $V_{r_i}$ to a different root-color space $V_{r_j}$,
then $S_0$ does not commute with $X$, because every operator commuting with $X$ preserves the
generalized eigenspace decomposition of $X$.

It remains to handle the case in which $S_0$ preserves every color space. Then the restriction of
$S_0$ to each color space is nilpotent, since $S_0^2=f(X)|_{V_0}$ is nilpotent there. Pick two
nonzero color spaces $V_{r_i}$ and $V_{r_j}$ with $i\ne j$. Because a nilpotent endomorphism of a
nonzero finite-dimensional space is not surjective, $\im(S_0|_{V_{r_i}})$ is a proper subspace,
and hence there is a nonzero functional $\ell$ on $V_{r_i}$ annihilating it.  Extend $\ell$ by zero on
all other color spaces.  Because a nilpotent endomorphism has nonzero kernel, choose $0\ne
w\in\ker(S_0|_{V_{r_j}})$. Define
\[
B(v)=\ell(v)w.
\]
The choices of $\ell$ and $w$ then give
\[
S_0B=0,
\qquad
BS_0=0,
\qquad
B^2=0.
\]
Therefore
\[
(S_0+B)^2=S_0^2=f(X)
\]
on $V_0$.  Put $Y_0=S_0+B$ on $V_0$.  The perturbation $B$ is nonzero and maps from color $i$ to
color $j$, so $Y_0$ does not preserve the $X$-primary color decomposition and hence does not commute
with $X$.  Extending $Y_0$ by any square root on the nonzero-primary part of $f(X)$ gives the
required global square root.
\end{proof}

\begin{remark}[Color-breaking deformations]
Theorem~\ref{thm:color-breaking-direction} says that once two root-colors are present in the
singular fiber, the square-root fiber has a color-breaking direction.  This direction may exist even
when a chosen square-root matching is entirely internal to colors.  Thus noncommutativity is not
merely a property of one bad matching: it is a geometric feature of the whole singular square-root
fiber.
\end{remark}

\subsection{The color graph}\label{subsec:colored-graph}

\begin{definition}[Color graph]\label{def:color-graph}
The color graph $\Gamma_0(X)$ has vertices given by the colored blocks $(p,i)\in\mathcal P_0(X)$.
Two vertices are joined when their sizes differ by at most one.  An edge is internal or
cross-colored according as its endpoints have the same or different colors.
\end{definition}

\begin{figure}[h]
\centering
\begin{tikzpicture}[scale=1.15,
  colorone/.style={circle,draw,thick,minimum size=13mm,font=\small,fill=blue!15},
  colortwo/.style={rectangle,draw,thick,minimum size=13mm,font=\small,fill=red!15},
  colorthree/.style={diamond,draw,thick,minimum size=13mm,aspect=1,font=\small,fill=green!15},
  internal/.style={dashed,thick},
  cross/.style={solid,very thick}]

\node[colorone]   (a) at (-2.6, 2.6)  {$(3,1)$};
\node[colorone]   (b) at ( 2.6, 2.6)  {$(2,1)$};
\node[colortwo]   (c) at ( 0,   1.05) {$(3,2)$};
\node[colorthree] (d) at ( 0,  -1.05) {$(2,3)$};
\node[colorthree] (e) at ( 0,  -3.0)  {$(1,3)$};

\draw[internal] (a) -- (b);
\draw[cross]    (a) -- (d);
\draw[cross]    (b) -- (d);
\draw[cross]    (a) -- (c);
\draw[cross]    (b) -- (c);
\draw[cross]    (c) -- (d);
\draw[cross]    (e) to[bend right=25] (b);
\draw[internal] (d) -- (e);

\node[align=left,font=\small,text width=11cm] at (0,-4.2)
{Circles, squares, and diamonds denote root-colors $1$, $2$, $3$.
Dashed edges preserve color, solid edges break color.};
\end{tikzpicture}
\caption{The color graph $\Gamma_0(X)$ for the running color pattern
$(3,1),(2,1),(3,2),(2,3),(1,3)$, drawn with all eight admissibility edges in a
planar embedding.  The two same-color (internal) edges $(3,1)\text{--}(2,1)$ and
$(2,3)\text{--}(1,3)$ are dashed.  The six cross-colored edges $(3,1)\text{--}(3,2)$,
$(2,1)\text{--}(3,2)$, $(3,1)\text{--}(2,3)$, $(2,1)\text{--}(2,3)$, $(3,2)\text{--}(2,3)$,
and $(2,1)\text{--}(1,3)$ are solid.  Each cross-colored edge is a color-breaking edge
and produces square roots that move between root-primary sectors and therefore do not
commute with $X$.
Vertices are shape-coded by root-color (circle, square, diamond) in addition to the
conventional fill colors, so the figure remains readable in monochrome print.  The two
non-edges $(3,1)\text{--}(1,3)$ and $(3,2)\text{--}(1,3)$ correspond to size differences
greater than one and are therefore inadmissible.}
\label{fig:colored-singular-fiber-graph}
\end{figure}

\begin{remark}[Combinatorial summary]
Figure~\ref{fig:colored-singular-fiber-graph} illustrates the convention.  The graph $\Gamma_0(X)$
separates three questions.  Existence asks whether the uncolored vertices admit a perfect matching,
allowing singleton vertices of size $1$.  Lemma~\ref{lem:two-chain-square-root-construction} gives
noncommutativity from cross-colored matched edges.  The color-breaking theorem says that, after
existence, the mere presence of at least two colors already guarantees a noncommuting direction,
even if the selected matching uses only internal edges.
\end{remark}

\begin{example}[A color-breaking example with internal and cross edges]\label{ex:color-breaking-colored}
Suppose the colored zero partition contains
\[
(3,1),\ (2,1),\ (3,2),\ (2,3),\ (1,3).
\]
The uncolored shadow is $(3,3,2,2,1)$, which is square-root pairable by matching $(3,3)$, $(2,2)$,
and leaving the singleton $(1)$.  If the two size-$3$ blocks have colors $1$ and $2$, the matching
$(3,1)\sim(3,2)$ is cross-colored and directly yields a noncommuting square root.  Even if one
chooses only internal pairings where possible, the presence of more than one color activates
Theorem~\ref{thm:color-breaking-direction}, so the square-root fiber still contains a color-breaking
deformation.
\end{example}

\begin{remark}[Effectiveness]\label{rem:effectiveness}
The classification is effective: given $X$, computing the Jordan partitions $\eta_{r_i}$, forming
the colored zero partition $\mathcal P_0(X)$, and testing whether at least two colors occur run in
polynomial time in $m$ over any field in which Jordan form and root extraction for $f$ are
polynomial-time computable.  Over $\overline{\mathbb F}_q$ in characteristic different from $2$ and
$3$ the colored data may be further refined by the Frobenius action on the roots of $f$
(cf.\ Section~\ref{sec:further-directions}).
\end{remark}

\section{Global classification after existence}\label{sec:global-classification-after-existence}

The next theorem packages the preceding noncommuting mechanisms into a single fiberwise
classification once existence of a square root is known.

\begin{lemma}[Polynomial recovery dichotomy]\label{lem:polynomial-recovery-dichotomy}
Let $A=f(X)$.  Then $X\in k[A]$ if and only if $f$ is injective on $\Spec(X)$ and, whenever
$f'(\lambda)=0$ for an eigenvalue $\lambda$ of $X$, the nilpotent part of $X$ on $V_\lambda$ is
zero.
\end{lemma}

\begin{proof}
With the primary decomposition $V=\bigoplus_{\lambda\in\Spec(X)}V_\lambda$ and
$X|_{V_\lambda}=\lambda I+N_\lambda$ from Section~\ref{sec:setup}, the operator $A=f(X)$ commutes
with $X$, so each $V_\lambda$ is stable under $A$.  On $V_\lambda$ one has
\[
A|_{V_\lambda}=f(\lambda)I+c_\lambda N_\lambda^{d_\lambda}U_\lambda(N_\lambda),
\]
where $d_\lambda$ is the order of vanishing of $f(t)-f(\lambda)$ at $t=\lambda$, $c_\lambda\ne 0$,
and $U_\lambda(0)\ne 0$.  Because $f$ is cubic with $\operatorname{char} k\ne 2,3$, the value
$d_\lambda$ lies in $\{1,2,3\}$, with $d_\lambda=1$ when $f'(\lambda)\ne 0$, $d_\lambda=2$ when
$f'(\lambda)=0$ but $f''(\lambda)\ne 0$, and $d_\lambda=3$ only when $f''(\lambda)=0$ as well, which
forces $f(t)=t^3+b$ and $\lambda=0$.  If $d_\lambda=1$, then $N_\lambda\in k[A|_{V_\lambda}]$ by inversion in the
truncated algebra, so $X|_{V_\lambda}\in k[A|_{V_\lambda}]$.  If $d_\lambda>1$ and $N_\lambda\ne 0$,
then every element of $k[A|_{V_\lambda}]$ is a polynomial in
$N_\lambda^{d_\lambda}U_\lambda(N_\lambda)$ and has no term with nonzero linear coefficient in
$N_\lambda$, and hence $N_\lambda\notin k[A|_{V_\lambda}]$, so local recovery fails.  If $N_\lambda=0$,
local recovery is trivial.

It remains to compare different primary summands.  If two distinct eigenvalues $\nu\ne\theta$
satisfy $f(\nu)=f(\theta)=\rho$, then the $\rho$-primary component of $A$ contains both $V_\nu$ and
$V_\theta$.  For every polynomial $h$, the operator $h(A)$ has the same scalar part $h(\rho)$ on
both quotients
\[
V_\nu/(A-\rho I)V_\nu
\qquad\text{and}\qquad
V_\theta/(A-\rho I)V_\theta .
\]
It therefore cannot induce multiplication by $\nu$ on the first quotient and by $\theta$ on the
second.  Hence $X\notin k[A]$ whenever such a spectral collision occurs.  Conversely, if the values
$f(\nu)$ are pairwise distinct and each local recovery condition holds, the local polynomials may be
patched by the Chinese remainder theorem for the pairwise coprime primary factors of the minimal
polynomial of $A$.  This gives a single polynomial $g$ with $g(A)=X$.
\end{proof}

\begin{lemma}[Spectral-collision idempotents]\label{lem:spectral-collision-idempotents}
Suppose $\lambda\ne\nu$ are eigenvalues of $X$ with $f(\lambda)=f(\nu)=\mu$, and let $W$ be the
$\mu$-primary component of $A=f(X)$.  Then there is an idempotent $P_W\in Z(A_W)\setminus Z(X_W)$,
and extending $P_W$ by zero on the other $A$-primary components gives an idempotent $P\in
Z(A)\setminus Z(X)$.  (The lemma is stated for all $\mu$.  It is invoked downstream only when
$\mu\ne 0$, where the localized idempotent twist of
Lemma~\ref{lem:localized-idempotent-twist} applies.)
\end{lemma}

\begin{proof}
Let $W$ be the $\mu$-primary component of $A$.  Since $A=f(X)$, the space $W$ is the direct sum of
the $X$-primary subspaces whose eigenvalues map to $\mu$.  In particular, it contains the
$X$-primary summands $V_\lambda$ and $V_\nu$, and all these summands are $A$-stable.

On $V_\lambda$ and $V_\nu$, set $z=A-\mu I$, which is nilpotent on both spaces.  Pick one
$z$-Jordan chain of length $r$ in $V_\lambda$ and one of length $s$ in $V_\nu$.  Viewing the two
chains as $k[z]/(z^r)$ and $k[z]/(z^s)$, the element $z^{\max\{s-r,0\}}\in k[z]/(z^s)$ is killed by
$z^r$, so the assignment $1\mapsto z^{\max\{s-r,0\}}$ defines a nonzero $k[z]$-module map.  Extend
by zero on the remaining $z$-Jordan chains of $V_\lambda$.  This gives a nonzero $k$-linear map
\[
\phi:V_\lambda\to V_\nu
\]
with $\phi A|_{V_\lambda}=A|_{V_\nu}\phi$.

On $V_\lambda\oplus V_\nu$, define
\[
P_W(u,v)=(u,\phi u).
\]
This is the projection onto the graph of $\phi$ along $V_\nu$.  Extend $P_W$ by zero on the other
$A$-stable summands of $W$.  Then $P_W^2=P_W$, and $P_W$ commutes with $A_W$ because $\phi$
intertwines the two restrictions of $A$.  Extending $P_W$ by zero on the other $A$-primary
components gives an idempotent $P\in Z(A)$.

It remains to check that $P_W$, and hence $P$, does not commute with $X$.  Write
\[
X|_{V_\lambda}=\lambda I+N_\lambda,
\qquad
X|_{V_\nu}=\nu I+N_\nu .
\]
The lower-left block of $P_WX_W-X_WP_W$ is
\[
\phi X|_{V_\lambda}-X|_{V_\nu}\phi=(\lambda-\nu)\phi+\phi N_\lambda-N_\nu\phi .
\]
The linear operator $\psi\mapsto \psi N_\lambda-N_\nu\psi$ on
$\operatorname{Hom}_k(V_\lambda,V_\nu)$ is nilpotent because left multiplication by $N_\nu$ and
right multiplication by $N_\lambda$ are commuting nilpotent operators.  Hence
$(\lambda-\nu)I+(\psi\mapsto \psi N_\lambda-N_\nu\psi)$ is invertible on
$\operatorname{Hom}_k(V_\lambda,V_\nu)$.  The displayed lower-left block is the value of this
invertible operator at $\phi$.  Since $\phi\ne 0$, it is nonzero, so $P_W\notin Z(X_W)$ and $P\notin
Z(X)$.
\end{proof}

\begin{lemma}[Critical-collapse idempotents]\label{lem:critical-collapse-idempotents}
Suppose $X=\lambda I+N$ on a generalized eigenspace $V_\lambda$, where $N\ne 0$ and
$f'(\lambda)=0$.  Let $W$ be the $f(\lambda)$-primary component of $A=f(X)$.  Then there exists an
idempotent $P_W\in Z(A_W)\setminus Z(X_W)$.  (Since $f$ is square-free, $f'(\lambda)=0$ already
forces $f(\lambda)\ne 0$, so $W$ is a nonzero-$A$-primary component.)
\end{lemma}

\begin{proof}
Locally,
\[
f(\lambda+t)-f(\lambda)=t^d u(t),\qquad d\in\{2,3\},\qquad u(0)\ne 0.
\]
In particular, $d\le 3$ because $f$ is cubic, and the nilpotency index of $N$ on $V_\lambda$ is
bounded by $\dim V_\lambda$, so the formal series $v$ below truncates after finitely many terms.
By Lemma~\ref{lem:critical-collapse-reduction}, applied with $H(N)=f(\lambda+N)-f(\lambda)$ and
$r=d$, there is a nilpotent matrix $M=Nv(N)$ such that $k[M]=k[N]$ and
\[
f(X)=f(\lambda)I+M^d.
\]
The same lemma also gives an idempotent $P_\lambda$ on $V_\lambda$ commuting with $M^d$, and hence
with $A|_{V_\lambda}$, but not with $N$ and therefore not with $X|_{V_\lambda}$.  The $A$-primary
component $W$ is a direct sum of $X$-primary subspaces whose eigenvalues map to $f(\lambda)$.  This
decomposition is $A$-stable because $A=f(X)$ preserves the primary decomposition of $X$.  Extending
$P_\lambda$ by zero on the remaining $X$-primary summands inside $W$ therefore still gives an
idempotent commuting with $A_W$.  Since its restriction to $V_\lambda$ does not commute with
$X|_{V_\lambda}$, the extended idempotent satisfies $P_W\in Z(A_W)\setminus Z(X_W)$.
\end{proof}

\begin{proof}[Proof of Theorem~\ref{thm:fiberwise-forced-commutativity-intro}]
If $X\in k[f(X)]$, then every square root of $f(X)$ commutes with $f(X)$ and therefore with $X$.
Conversely, assume $X\notin k[f(X)]$.  By Lemma~\ref{lem:polynomial-recovery-dichotomy}, either there
is a spectral collision or there is a nontrivial critical block, and these are the only ways polynomial
recovery can fail.

A zero-valued collision gives a multicolor singular zero-primary component, so
Theorem~\ref{thm:color-breaking-direction} gives a noncommuting square root.  A nonzero-valued
collision gives an idempotent $P_W\in Z(A_W)\setminus Z(X_W)$ on the corresponding nonzero
$A$-primary component by Lemma~\ref{lem:spectral-collision-idempotents}, and since the global
square-root fiber is nonempty by hypothesis, the localized idempotent twist,
Lemma~\ref{lem:localized-idempotent-twist}, gives a noncommuting square root.  Finally, a nontrivial
critical block gives an idempotent on the corresponding $A$-primary component by
Lemma~\ref{lem:critical-collapse-idempotents}.  A critical point cannot be a root because $f$ is
square-free, so its critical value is nonzero.  Thus the relevant component is a nonzero $A$-primary
component, and Lemma~\ref{lem:localized-idempotent-twist} again gives a noncommuting square root.
\end{proof}

\begin{remark}[Role of the cross-root criterion]
Cross-root singular noncommutativity is a useful explicit subcase of
Theorem~\ref{thm:color-breaking-direction}.  The rank-one deformation theorem above removes the need
to assume that a chosen admissible square-root pairing already crosses colors.
\end{remark}

\begin{remark}[Higher-degree polynomials]\label{rem:higher-degree}
The proof structure above does not use the cubic degree of $f$ in any essential way beyond two
points: the explicit critical-collapse exponent $d\in\{2,3\}$ in
Lemma~\ref{lem:critical-collapse-idempotents}, and the explicit cross-root counting in
Section~\ref{sec:colored-singular-fibers}, where the number of root-colors is $\deg f$.  For an
arbitrary square-free polynomial $f\in k[t]$ of degree at least $2$ over an algebraically closed
field with $\operatorname{char} k\nmid d!$ for every critical multiplicity $d$, the same dichotomy
holds: when the square-root fiber $\{Y:Y^2=f(X)\}$ is nonempty, every such $Y$ commutes with $X$
if and only if $X\in k[f(X)]$.  The cubic case is the smallest setting in which all three
mechanisms (spectral collision, critical collapse, multicolor zero fiber) appear, and the result
is stated in this minimal form to keep notation transparent.  In particular, the
degenerate case $\deg f=2$ does not exhibit the full dichotomy: a square-free quadratic has at
most one critical point and exactly two roots, so the multicolor zero-fiber and critical-collapse
mechanisms cannot interact in a way that activates all three obstruction types simultaneously.
\end{remark}

\begin{theorem}[Over a general field]\label{thm:general-field}
Let $K$ be a field with $\operatorname{char} K\ne 2,3$, let $f\in K[t]$ be a square-free cubic, and
let $X\in M_m(K)$.  Set $A=f(X)$.  Assume that the equation $Y^2=A$ has a solution over $K$.  If
$X\in K[A]$, then every $K$-rational square root of $A$ commutes with $X$.  If $X\notin K[A]$ and
in addition every square root of $A$ in $M_m(\bar K)$ is already defined over $K$ (e.g., when $A$
is split over $K$ in the sense of Higham \cite{HighamFunctions}), then the fiber
$\{Y\in M_m(K):Y^2=A\}$ contains an element $Y$ with $[X,Y]\ne 0$.
\end{theorem}

\begin{proof}
The forward direction is identical to the proof of
Theorem~\ref{thm:fiberwise-forced-commutativity-intro} above: if $X\in K[A]$, then
Lemma~\ref{lem:centralizer-containment} forces every square root in $Z(A)\subseteq Z(X)$.  For the
converse, the constructions in Section~\ref{sec:colored-singular-fibers} produce a noncommuting
square root $Y$ in $M_m(\bar K)$.  Under
the rationality hypothesis on the square-root fiber, this $Y$ lies in $M_m(K)$.  Without the
rationality hypothesis the construction may produce a square root defined only over a quadratic
extension.  The obstruction to rationality is then governed by the splitting field of the
zero-primary nilpotent part of $A$ and is unrelated to the commutativity question.
\end{proof}

\section{Further directions}\label{sec:further-directions}

The classification suggests two follow-up problems, neither of which is used in the proof above.
The first is arithmetic.  Over a finite field $\mathbb F_q$ of characteristic different from $2$ and
$3$, for a square-free cubic $f\in\mathbb F_q[t]$, one can compare the full count
\[
N_m^{\mathrm{all}}(q)=\#\{(X,Y)\in M_m(\mathbb F_q)^2:Y^2=f(X)\}
\]
with its commuting part
\[
N_m^{\mathrm{com}}(q)=\#\{(X,Y):Y^2=f(X),\ [X,Y]=0\},
\qquad
N_m^{\mathrm{nc}}(q)=N_m^{\mathrm{all}}(q)-N_m^{\mathrm{com}}(q).
\]
Existing finite-field matrix-point counts on elliptic curves usually impose commutativity on the
matrix coordinates \cite{AllenDavidLennonOnoSaad,BlaserBradleyVargasXing}.  The results above
suggest refining the noncommuting excess $N_m^{\mathrm{nc}}(q)$ by the Frobenius action on the roots
of $f$ and by the colored Jordan data of the zero-primary sector.  As a concrete consequence
of the main theorem, every $X\in M_m(\mathbb F_q)$ in the polynomial-recovery locus contributes
zero to $N_m^{\mathrm{nc}}(q)$, so the only contributions arise from $X$ realizing a spectral
collision, a critical collapse, or a multicolor zero-primary fiber, and the Frobenius action on the
roots of $f$ controls which of these mechanisms is available over $\mathbb F_q$ as opposed to over
the algebraic closure.

Over $\mathbb C$ the analogous question for self-adjoint $X$ reduces, via the spectral theorem, to a condition on the eigenspace lattice of $X$ relative to $A$ and does not produce new mechanisms beyond those classified above.

\begin{remark}[Cuspidal and nodal degenerations]\label{rem:cuspidal-nodal}
The paper assumes $f$ square-free, equivalently $4a^3+27b^2\ne 0$.  Allowing $f$ to acquire a
repeated root degenerates the scalar curve $y^2=f(x)$ to a cuspidal or nodal cubic, depending on
whether $f$ has a triple or double root.  Such degenerations break some of the auxiliary lemmas
(for instance, $f'$ acquires a common factor with $f$, and the critical-collapse exponent
$d\in\{2,3\}$ is no longer the local order of vanishing of $f(t)-f(\lambda)$ in the relevant
sense).  The classification statement still has a meaningful analogue, but the explicit witness
constructions require additional care.  The non-square-free case is therefore left for separate
treatment.
\end{remark}

\section{Relation to existing work}\label{sec:related-work-novelty}

The equation studied in this paper is the matrix version of a short-Weierstrass elliptic cubic
$E_f:y^2=f(x)$, with matrix variables substituted for the scalar coordinates and without imposing
commutativity. This convention is standard for the scalar curve when $f$ is a square-free cubic in
characteristic different from $2$ and $3$ \cite{SilvermanArithmetic}. The point of comparison is
important: many ingredients below are classical, while the present contribution is the fiberwise
commutativity classification for the matrix equation $Y^2=f(X)$.

\subsection*{Scope of the contribution}
The square-root existence criterion, the Jordan-form reductions, and the polynomial functional
calculus used below are standard inputs.  The new point is the relative question: after a square
root of $f(X)$ exists, does the entire square-root fiber lie in $Z(X)$?

\subsection*{The representation variety of $\nu_f$ and matrix factorizations}
In algebraic language, a solution pair $(X,Y)\in\mathcal M_m(f)$ is a finite-dimensional
representation of the noncommutative $k$-algebra
\[
 \nu_f=k\langle x,y\rangle/(y^2-f(x)),
\]
and the commuting locus $\mathcal C_m(f)$ corresponds to representations that factor through the
abelianization $k[x,y]/(y^2-f(x))$, the affine coordinate ring of the elliptic curve $E_f$.  The
main theorem describes the obstruction to such factorization in terms of polynomial recovery,
placing the result in the representation-variety setting natural to the Journal of Algebra.  The
present setup is distinct from matrix factorizations of a polynomial in the sense of
\cite{EisenbudMF}, where one studies polynomial-entry matrices $\Phi,\Psi$ with $\Phi\Psi=\Psi\Phi=fI$,
and from polynomial-matrix factorization problems such as \cite{PolynomialMatrixFactorizations}.
It is also separate from the theory of free and noncommutative functions evaluated on tuples of
matrices \cite{AglerMcCarthyNCFunctions}, and from Sklyanin-algebra representation problems
associated with elliptic data \cite{WaltonSklyanin}.  Here a scalar polynomial $f$ is evaluated at
one finite-dimensional matrix $X$, and the problem is to classify when the fiber $\{Y:Y^2=f(X)\}$
is contained in $Z(X)$.  Equivalently, the main theorem describes the locus of points $X\in
M_m(k[t])$ over which the fiber of the projection $\mathcal M_m(f)\to M_m(k[t])$, $(X,Y)\mapsto X$,
lies entirely in the abelianization stratum $\mathcal C_m(f)$, namely the polynomial-recovery
locus $\{X:X\in k[f(X)]\}$.

\subsection*{Matrix roots and the $f(X)=A$ lineage}\label{subsec:fX-equals-A-lineage}
The existence theory for square roots of a fixed matrix has standard references
\cite{HighamFunctions,JohnsonOkuboReams,BjorkHammarling,HornJohnsonMatrixAnalysis,Gantmacher} and
modern structured variants \cite{HNonnegativeSquareRoots}: nonsingular Jordan blocks admit square
roots by finite binomial expansions, and singular existence is governed by the nilpotent
Jordan-block pairing criterion.  These results are inputs to this paper.  Centralizers and
polynomial-equivalence questions provide the natural language for recovery
\cite{XiZhangCentralizer,ZhangZhuPolynomialEquivalence}.  Lemma~\ref{lem:centralizer-containment}
gives $Y\in Z(f(X))$, so the distinction between $Z(X)$ and $Z(f(X))$ controls the possibility of
noncommutativity.  More broadly, a related body of work studies the equation $f(X)=A$ with $A$
fixed and $X$ as the unknown: Sylvester's interpolation formula, Frobenius's commutativity
criterion \cite{Gantmacher,HighamFunctions}, the polynomial-in-$A$ solvability criterion of
W.\,E. Roth \cite{Roth1928} for $P(X)=A$, Thompson's classification of $B=f(A)$ and the
centralizer-equality theorem \cite{Thompson1969}, the explicit complex-square-root constructions of
Cross and Lancaster \cite{CrossLancaster1974}, the generalized-inverse analysis of Drazin
\cite{Drazin2007}, Wadsworth's commuting-pair structure theorem \cite{Wadsworth1990}, and the
recent explicit construction of solutions to $p(X)=A$ over extensions of $\mathbb Q$ by
Groenewald et al.\ \cite{Groenewald2022}.  All of these fix a coefficient matrix and study which
matrices satisfy a polynomial constraint, typically restricting to commuting solutions or to
structural properties of the solution variety.  The present paper inverts this perspective: $X$
is the input, $A=f(X)$ is determined, the fiber $\{Y:Y^2=A\}$ is the object of study, and the
question is whether the entire fiber lies in $Z(X)$.  Theorem~\ref{thm:general-field} extends the
conclusion to general fields, complementing the algebraically closed centralizer setting of
Thompson, and the constructive noncommuting square roots produced here are not directly available
from the centralizer-equality literature.

\subsection*{Adjacent equations, matrix points, and Jordan strata}
There is related work on commuting and noncommuting solutions of other quadratic matrix equations,
including Yang--Baxter type equations and $AXA=XAX$
\cite{DongDingQuadratic,LuYangBaxterSingular}, which typically fix a coefficient matrix and study
a different equation under Jordan-form, rank, or spectral hypotheses on that matrix.  Neighboring
arithmetic work counts matrix points on algebraic varieties over finite fields
\cite{AllenDavidLennonOnoSaad,BlaserBradleyVargasXing}. In the elliptic-curve case this literature
counts commuting matrix analogues of $B^2=A(A-I_n)(A-aI_n)$ with an explicit commutativity
condition on the coordinates, whereas this paper does not impose $[X,Y]=0$ and instead
studies when the equation itself forces commutativity, with a finite-field variant asking for the
noncommuting excess $N_m^{\mathrm{nc}}(q)$.  Finally, the Jordan-type stratification of nilpotent
centralizers is well developed
\cite{BoijIarrobinoKhatamiJordanStrata,IarrobinoKhatamiVanSteirteghemZhaoNilpotentTypes}.  Those
works study possible and generic Jordan types inside nilpotent centralizers as intrinsic invariants
of a single nilpotent operator, whereas the colored zero partitions used here are extrinsic data
attached to the pair $(X,f)$: each block of the nilpotent operator $f(X)|_{V_0}$ carries the label
of the $X$-primary sector $V_{r_i}$ that produced it, and the labeling is determined by the
polynomial $f$ via its root data, not by the nilpotent operator $f(X)|_{V_0}$ alone.

\subsection*{Precise contribution}
The result is best viewed as a relative classification over the polynomial map $X\mapsto f(X)$.  It
is not a new criterion for square roots of a single matrix, nor a count of commuting matrix points,
nor a general representation-variety classification.  Its content is that the full fiber
$\{Y:Y^2=f(X)\}$ is contained in $Z(X)$ exactly in the polynomial-recovery case, and that the proof
gives explicit witnesses outside that case.

\section{Examples}\label{sec:examples}

The first example illustrates the recovery side of the theorem.  The remaining examples show the
basic ways in which polynomial recovery can fail and a noncommuting square root can appear: nonzero
spectral collision, singular cross-root collision, critical nilpotent collapse, and pure inflection
collapse.

\begin{example}[A forced-commutativity example]\label{ex:forced-commutativity}
Let $f(t)=t^3+1$ and let
\[
X=\begin{pmatrix}1&1\\0&1\end{pmatrix}.
\]
Write $N=X-I$, so $N^2=0$.  Since $\operatorname{char} k\ne2,3$, the values $f(1)=2$ and $f'(1)=3$
are nonzero.  If $A=f(X)$, then
\[
A=2I+3N,
\qquad
X=I+\frac{1}{3}(A-2I)\in k[A].
\]
The matrix $A$ is invertible, so the fiber $\{Y:Y^2=A\}$ is nonempty.  Therefore every solution of
$Y^2=f(X)$ commutes with $X$, by Theorem~\ref{thm:fiberwise-forced-commutativity-intro}.
\end{example}

\begin{example}[Nonzero spectral-collision $2\times 2$ example]\label{ex:nonzero-spectral-collision-2x2}
Let $f(t)=t^3-3t$ and
\[
X=\begin{pmatrix}-1&0\\0&2\end{pmatrix}.
\]
Then $f(-1)=f(2)=2$, so $f(X)=2I$.  Choose $\alpha$ with $\alpha^2=2$ and set
\[
Y=\alpha\begin{pmatrix}0&1\\1&0\end{pmatrix}.
\]
Then $Y^2=f(X)$, but $[X,Y]\ne 0$.  This is the basic nonzero spectral-collision mechanism.
\end{example}

\begin{example}[Singular cross-root $2\times 2$ example]\label{ex:singular-cross-root-2x2}
Let $f(t)=t^3-t$ and
\[
X=\begin{pmatrix}0&0\\0&1\end{pmatrix},\qquad
Y=\begin{pmatrix}0&1\\0&0\end{pmatrix}.
\]
Then $f(X)=0$, $Y^2=0$, and $[X,Y]\ne 0$.  This is the smallest cross-root singular mechanism: two
distinct roots of $f$ collapse to the same zero value of $f(X)$.
\end{example}

\begin{example}[Simple critical-collapse, $d_\lambda=2$]\label{ex:simple-critical-collapse}
Let $f(t)=t^3-3t$.  Then $f'(t)=3(t^2-1)$ has simple roots at $t=\pm 1$, with $f''(\pm 1)=\pm 6\ne 0$,
so both critical points are simple.  Take $\lambda=1$, so $f(1)=-2$, $f'(1)=0$, $f''(1)=6$.  Let
\[
X=\begin{pmatrix}1&1\\0&1\end{pmatrix}=I+N,\qquad N=\begin{pmatrix}0&1\\0&0\end{pmatrix},\quad N^2=0.
\]
Taylor expansion gives $f(I+N)=f(1)I+f'(1)N+\tfrac{f''(1)}{2}N^2=-2I_2$, using $N^2=0$.  Choose
$\alpha\in k$ with $\alpha^2=-2$ and set
\[
Y=\alpha\begin{pmatrix}0&1\\1&0\end{pmatrix}.
\]
Then $Y^2=\alpha^2 I_2=-2I_2=f(X)$, and a direct calculation gives
\[
[X,Y]=\begin{pmatrix}\alpha&0\\0&-\alpha\end{pmatrix}\ne 0.
\]
This activates the critical-collapse mechanism with $d_\lambda=2$: the spectrum is
$\Spec(X)=\{1\}$, $f$ vanishes to order $2$ in $f(t)-f(1)=(t-1)^2(t+2)$ at $t=1$, no spectral
collision is available (only one eigenvalue), and the singular zero-primary fiber is empty.  By
Lemma~\ref{lem:critical-collapse-idempotents}, the construction realizes the idempotent twist on
the nonzero $A$-primary component at $f(1)=-2$.
\end{example}

\begin{example}[Pure inflection collapse]\label{ex:pure-inflection-collapse}
Let $f(t)=t^3+1$ and $X=J_3(0)$.  Then $f(X)=I$.  The diagonal involution
\[
Y=\operatorname{diag}(1,-1,1)
\]
satisfies $Y^2=I=f(X)$ and does not commute with $X$.  This is a pure polynomial-recovery failure:
$f(X)$ is scalar while $X$ retains nilpotent data.
\end{example}

\begin{example}[A $5\times5$ nonzero spectral collision]\label{ex:5x5-nonzero-spectral-collision}
Let $f(t)=t^3+1$, let $k$ be algebraically closed with $\operatorname{char} k\ne 2,3$, and let
$\omega\in k$ be a primitive cube root of unity.  Set
\[
X=\operatorname{diag}(1,\,1,\,\omega,\,\omega^2,\,\omega^2)\in M_5(k).
\]
The spectrum is $\Spec(X)=\{1,\omega,\omega^2\}$.  Since $1^3=\omega^3=(\omega^2)^3=1$, one has
$f(\lambda)=2$ for every $\lambda\in\Spec(X)$, hence
\[
A=f(X)=2I_5,
\]
a single $A$-primary component at the nonzero value $2$.  Lemma~\ref{lem:poly-square-root} gives a
polynomial square root, so the fiber $\{Y\in M_5(k):Y^2=A\}$ is nonempty.  Since
$f|_{\Spec(X)}$ is constant but $X$ is not scalar, $X\notin k[A]$ by
Lemma~\ref{lem:polynomial-recovery-dichotomy}.  The main theorem therefore predicts a noncommuting
square root.

An explicit witness is constructed as follows.  Let $\alpha\in k$ satisfy $\alpha^2=2$ and let
$Q\in M_5(k)$ be the permutation matrix exchanging the basis vectors $e_1$ and $e_3$ and fixing
$e_2$, $e_4$, $e_5$.  Then $Q^2=I_5$ and $Q\notin Z(X)$.  Put $Y=\alpha Q$.  Then
\[
Y^2=\alpha^2 Q^2=2 I_5=A.
\]
To verify $[X,Y]\ne 0$, compute
\[
XQ(e_1)=X(e_3)=\omega\,e_3,\qquad QX(e_1)=Q(e_1)=e_3,
\]
so $[X,Y](e_1)=\alpha(\omega-1)e_3\ne 0$.  Conceptually, $Y=\alpha Q$ is the spectral-collision
idempotent twist (Lemma~\ref{lem:spectral-collision-idempotents}) applied to the pair of distinct
eigenvalues $1,\omega\in\Spec(X)$ inside the single $A$-primary component at the value $2$: the
idempotent $P=(I_5+Q)/2$ lies in $Z(A)\setminus Z(X)$, and the twist $R=2P-I_5=Q$ produces the
required noncommuting square root by Lemma~\ref{lem:idempotent-twist} with $G=\alpha I_5\in k[A]$.
The zero-primary component of $A$ is empty, so the witness is purely of
nonzero-spectral-collision type, in contrast with the $2\times 2$ singular cross-root
Example~\ref{ex:singular-cross-root-2x2}.
\end{example}

\begin{example}[A compound example: critical collapse and spectral collision coexist]\label{ex:compound-critical-and-collision}
Let $f(t)=t^3-3t$ and
\[
X=\begin{pmatrix}1&1&0\\0&1&0\\0&0&-2\end{pmatrix}=J_2(1)\oplus J_1(-2)\in M_3(k).
\]
The spectrum is $\Spec(X)=\{1,-2\}$ and $f(t)-f(1)=(t-1)^2(t+2)$, so $\lambda=1$ is a simple critical point with $d_1=2$.  Writing $X|_{V_1}=I+N$ with $N=\begin{psmallmatrix}0&1\\0&0\end{psmallmatrix}$ and $N^2=0$, Taylor expansion gives $f(I+N)=-2I_2$.  On the eigenspace $V_{-2}$, the operator $X$ acts as the scalar $-2$, with $f(-2)=-2=f(1)$.  Therefore
\[
A=f(X)=-2 I_3.
\]
The analysis identifies two distinct mechanisms in $X$:
\begin{enumerate}
\item[(i)] critical collapse at $\lambda=1$ with $d_1=2$, since $f'(1)=0$ and $X|_{V_1}$ is a $J_2$-block (Lemma~\ref{lem:critical-collapse-idempotents}),
\item[(ii)] spectral collision between $\lambda=1$ and $\lambda=-2$, since $f(1)=f(-2)$ and the $A$-Jordan structure on $V_{-2}$ is a size-$1$ block compatible with the $(1,1)$ component of $A|_{V_1}$ (Lemma~\ref{lem:spectral-collision-idempotents}).
\end{enumerate}
Because $A=-2I_3$ is scalar but $X$ is not, polynomial recovery fails (Lemma~\ref{lem:polynomial-recovery-dichotomy}), and the main theorem predicts a noncommuting square root.  Each mechanism produces its own witness.

Let $\alpha\in k$ satisfy $\alpha^2=-2$. The critical-collapse witness, supported on $V_1$, is
\[
Y_{\mathrm{cc}}=\begin{pmatrix}0&\alpha&0\\\alpha&0&0\\0&0&\alpha\end{pmatrix},\qquad Y_{\mathrm{cc}}^2=-2I_3=A,\qquad [X,Y_{\mathrm{cc}}]=\begin{pmatrix}\alpha&0&0\\0&-\alpha&0\\0&0&0\end{pmatrix}\ne 0.
\]
The spectral-collision witness, mixing $V_1$ and $V_{-2}$, is
\[
Y_{\mathrm{sc}}=\begin{pmatrix}0&0&\alpha\\0&\alpha&0\\\alpha&0&0\end{pmatrix},\qquad Y_{\mathrm{sc}}^2=-2I_3=A,\qquad [X,Y_{\mathrm{sc}}]=\begin{pmatrix}0&\alpha&3\alpha\\0&0&0\\-3\alpha&-\alpha&0\end{pmatrix}\ne 0.
\]
The witness $Y_{\mathrm{cc}}$ is the idempotent twist of Lemma~\ref{lem:critical-collapse-idempotents} applied at $\lambda=1$, extended by a commuting square root on $V_{-2}$.  The witness $Y_{\mathrm{sc}}$ is the idempotent twist of Lemma~\ref{lem:spectral-collision-idempotents} applied to the pair $\{1,-2\}$, extended by a commuting square root on a chosen complementary line in $V_1$.  Both witnesses certify the same forced-commutativity failure but arise from distinct lemmas of the paper. This example shows that the classification cleanly separates the two mechanisms even when they coexist in the same $X$.
\end{example}

\end{document}